\documentclass{amsart}

\usepackage[latin1]{inputenc}
\usepackage{amsfonts}
\usepackage{amsmath}
\usepackage{amsthm}
\usepackage{amssymb}
\usepackage{latexsym}
\usepackage{enumerate}

\newtheorem{theorem}{Theorem}%[section]
\newtheorem{lemma}[theorem]{Lemma}

\newtheorem{corollary}[theorem]{Corollary}

\newtheorem{claim}[theorem]{Claim}

\newtheoremstyle{note}% <name>
{3pt}% <Space above>
{3pt}% <Space below>
{}% <Body font>
{}% <Indent amount>
{\bf}% <Theorem head font>
{:}% <Punctuation after theorem head>
{.5em}% <Space after theorem headi>
{}% <Theorem head spec (can be left empty, meaning `normal')>

\theoremstyle{note}
\newtheorem{remark}[theorem]{Remark}

\numberwithin{equation}{section}

\begin{document}

\newcommand{\cc}{\mathfrak{c}}
\newcommand{\N}{\mathbb{N}}
\newcommand{\BB}{\mathbb{B}}
\newcommand{\C}{\mathbb{C}}
\newcommand{\Q}{\mathbb{Q}}
\newcommand{\R}{\mathbb{R}}
\newcommand{\Z}{\mathcal{Z}}
\newcommand{\T}{\mathbb{T}}
\newcommand{\st}{*}
\newcommand{\PP}{\mathbb{P}}
\newcommand{\rin}{\right\rangle}
\newcommand{\SSS}{\mathbb{S}}
\newcommand{\forces}{\Vdash}
\newcommand{\dom}{\text{dom}}
\newcommand{\osc}{\text{osc}}
\newcommand{\F}{\mathcal{F}}
\newcommand{\A}{\mathcal{A}}
\newcommand{\B}{\mathcal{B}}
\newcommand{\I}{\mathcal{I}}
\newcommand{\X}{\mathcal{X}}
\newcommand{\Y}{\mathcal{Y}}
\newcommand{\CC}{\mathcal{C}}

\thanks{The first named author was partially supported by the NCN (National Science
Centre, Poland) research grant no.\ 2020/37/B/ST1/02613.}

\title{Large Banach spaces with no infinite equilateral sets}
\author{Piotr Koszmider}
\address{Institute of Mathematics of the Polish Academy of Sciences,
ul. \'Sniadeckich 8,  00-656 Warszawa, Poland}
\email{\texttt{piotr.koszmider@impan.pl}}

\author{Hugh Wark}
\address{York, North Yorkshire, England}
\email{\texttt{hughwark@yahoo.com}}

\begin{abstract} 
A subset of a Banach space is called equilateral if the distances between any two of its distinct elements are the same. It is proved that there exist non-separable Banach spaces (in fact of density continuum) with no infinite equilateral subset. These examples are strictly convex renormings of $\ell_1([0,1])$. A wider class of renormings of $\ell_1([0,1])$  which admit no uncountable equilateral sets is also considered.
\end{abstract}

\maketitle

\section{introduction}

Let $\X$ be a Banach space and $r>0$ a real. A subset $\Y\subseteq \X$ is called $r$-equilateral if
$\|x-y\|=r$ for any two distinct $x, y\in \Y$; it is called equilateral if it is $r$-equilateral
for some real $r>0$. 

As shown by Brass and by Dekster (\cite{brass, dekster}) for each $k\in \N\setminus\{0,1\}$ 
there is $d(k)\in \N$
such that every normed space of dimension $d(k)$ admits a $k$ element $r$-equilateral set.
However, the smallest value of $d(k)$ is unknown and it is an open conjecture that $d(k)$
 can take the value $k-1$ for each $k\in \N\setminus\{0,1\}$  (\cite{thompson}). 
 
 The above results
 of \cite{brass, dekster} imply that any infinite dimensional Banach space contains arbitrarily
 large finite equilateral sets. In fact,
 by a result of Shkarin (\cite{shkarin}) every finite
 ultrametric space (a metric space where the distance $d$ satisfies $d(x, z)\leq\max(d(x,y), d(y,z))$
 for any points $x, y, z$) isometrically embeds in any infinite dimensional
 Banach space. A surprising result  was obtained by Terenzi who proved
in \cite{terenzi} that there are 
infinite dimensional (separable) Banach
spaces with no infinite equilateral sets (for other spaces with this
property see \cite{glau-inf, terenzi2}). On the other 
hand  Mercourakis and Vassiliadis proved that  
 any Banach space containing
an isomorphic copy of $c_0$ admits an
infinite equilateral set (\cite{mer-pams})  and Freeman,  Odell,  Sari and  Schlumprecht 
proved that every uniformly smooth Banach space admits an infinite equilateral set (\cite{smooth}).
The difference between the example of Terenzi and the latter Banach spaces should be seen not
only in the context of the geometry of Banach spaces but also in the context of
infinite combinatorics, in particular the applicability of Ramsey methods in Banach spaces.

A natural problem if every nonseparable Banach space admits an uncountable 
equilateral set   has been considered in \cite{mer-pams, mer-c, equi}. 
The first named author constructed in \cite{equi}
a consistent example of a nonseparable Banach space 
which does not admit an uncountable equilateral set. It is of the form
$C(K)$, where $K$ is Hausdorff and compact. However,  it was also
proved in \cite{equi}  that it is consistent that no such Banach space
of the form $C(K)$ exists. This showed that 
the problem whether a nonseparable Banach space of the form $C(K)$ 
admits an uncountable equilateral set is undecidable (\cite{equi}). 

The main result of this paper is that there are absolute (under no extra set-theoretic assumption)
examples of nonseparable Banach spaces with no uncountable equilateral sets (necessarily not of
the form $C(K)$).
Moreover, they do not even admit an infinite equilateral sets and have density 
continuum\footnote{We do not know if the density continuum is the maximal possible. The only result bounding the densities of Banach spaces with no uncountable equilateral sets was obtained  by  Terenzi in \cite{terenzi} using essentially an Erd\"os-Rado type argument which is an uncountable version 
of the Ramsey theorem:
if the density of a Banach space $\X$ is bigger than  $2^{(2^\omega)}$, then
$\X$ admits an uncountable equilateral set.}.

Our approach is to transfer some parts of the Terenzi arguments from \cite{terenzi} to the nonseparable
setting. He considered a renorming of $\ell_1$, where the norm is given for any $x\in \ell_1$ by 
$$\|x\|=\|x\|_1+\sqrt{\sum_{i\in \N}{x(i)^2\over 2^i}}.$$
We consider renormings of $\ell_1([0,1])$ where the norm is 
defined for any $x\in \ell_1([0,1])$ by
$$\|x\|_T=\|x\|_1+\|T(x)\|_\X, \leqno  (\odot)$$
where $T:\ell_1([0,1])\rightarrow\X$ is an injective operator into a Banach space $\X$.
Renormings of $\ell_1([0,1])$ similar but different to ours were already employed in e.g., \cite{godun}
to obtain Banach spaces not admitting certain subsets. The foundation of our main result is the following:

\begin{theorem}\label{infinite} Suppose that $\X$ is a Banach space with a strictly convex norm and 
that $T:\ell_1([0,1])\rightarrow \X$ is a compact bounded injective operator.
Then the equivalent renorming $(\ell_1([0,1]), \|\ \|_T)$
of $(\ell_1([0,1]), \|\ \|_1)$ 
admits no infinite equilateral set.
\end{theorem}

Since there exist operators as in the hypothesis of the above theorem
(Lemma \ref{operators}) we obtain:

\begin{corollary}\label{main} There is a Banach space of density continuum which does not admit
an infinite equilateral set.
\end{corollary} 

In particular, this solves the question of whether there is a nonseparable
Banach space with no uncountable equilateral set (\cite{mer-c, equi}, Problem 293 of  \cite{guirao}).
Another absolute construction of a nonseparable Banach space with no uncountable
equilateral set is being presented at the same time in a paper by the first named author
\cite{kottman}. However, that is a renorming of
a space $C_0(K)$ for $K$ locally compact and scattered, so it is $c_0$-saturated (by \cite{pel}).
Since a result in \cite{mer-pams} says that any Banach space which contains an isomorphic copy
of $c_0$ admits an infinite equilateral set, we conclude that
spaces of \cite{kottman} admit such infinite sets.

By an argument of Terenzi from \cite{terenzi} given any  equilateral set $\Y$ in
a Banach space $\X$ we may assume that it is a $1$-equilateral set by scaling it.
Considering $\{y-y_0: y\in \Y\setminus\{y_0\}\}$ for any $y_0\in \Y$
we may assume that it is a $1$-equilateral set included in the unit sphere of $\X$.
Thus equilateral sets are related to the questions concerning separation of points
in the spheres of Banach spaces (see e.g. \cite{hajek-tams} for references).
Recall that a subset $\Y$ of a Banach space $\X$ is called
$\delta$-separated if $\|y-y'\|\geq \delta$ for all distinct $y, y'\in \Y$.
It is called $(\delta+)$-separated if  $\|y-y'\|> \delta$ for all distinct $y, y'\in \Y$.
By Remark 3.16  \cite{hajek-tams} the unit sphere of every renorming of $\ell_1([0,1])$
 contains  a subset $\Y$ of cardinality continuum such that $\|y-y'\|\geq1+\varepsilon$
 some $\varepsilon>0$ and for every two distinct $y, y'\in \Y$. 
 
 After proving Theorem \ref{infinite} in Section 3 we consider renormings 
 $\|\ \|_T$ of $\ell_1([0,1])$ as in ($\odot$) for any injective $T$  with separable range.
 Some of such renormings  admit many infinite equilateral sets (see Remark \ref{plus-id}). We obtain:
 
\begin{theorem}\label{uncountable}
Suppose that $\X$ is a Banach space and  $T:\ell_1([0,1])\rightarrow \X$ is injective 
and has separable range and $r>0$.
Then $(\ell_1([0,1]), \|\ \|_T)$
 has the following property: Any $r$-separated subset  $\Y\subseteq \ell_1([0,1])$ of regular uncountable cardinality has a subset $\Z\subseteq \Y$ of the same cardinality which is $(r+)$-separated.
In particular $(\ell_1([0,1]), \|\ \|_T)$ does not admit any uncountable equilateral set.
\end{theorem}

The above property  for renormings induced by
 $T$ injective compact and with strictly convex range is a  consequence of Theorem \ref{infinite} and
 some  partition calculus results
 (see Corollary \ref{dushnik}). Also this property is much stronger
 than not having uncountable equilateral sets (see Remark \ref{sierpinski}).
 
 A close link
 between separated subsets in the sphere and Auerbach bases
 was  demonstrated in \cite{hajek-tams}. In fact Godun's renorming of $\ell_1([0,1])$
was designed to prove that the space has no fundamental Auerbach system  (\cite{godun}).
Nevertheless, we do not know if our spaces admits an uncountable Auerbach system.

Let us also remark that considering Banach spaces without large equilateral sets which have renormings
admitting large equilateral sets (obviously $\ell_1([0,1])$ admits equilateral set of
cardinality continuum in the standard norm) is sometimes necessary to obtain
examples of the former kind. For
example, by a result of Swanepoel (\cite{swan}) any infinite dimensional
Banach space has an equivalent renorming which admits an infinite equilateral set (see also
\cite{mer-pams}). Moreover by the results  of \cite{mer-pams} the existence
of a biorthogonal system of cardinality $\kappa$ in a Banach space $\X$  implies
the existence of an equivalent renorming of $\X$ which admits equilateral set of cardinality $\kappa$.
This means by a result of Todorcevic (\cite{stevo-biorth}) that it is consistent
that every nonseparable Banach space has an equivalent renorming which admits an uncountable
equilateral set. We do not know however if it is consistent that there is
a nonseparable Banach space without an equivalent renorming which admits uncountable
equilateral sets. The densities of such an example could not exceed  continuum
(By a result of W. Johnson that any Banach space of density
bigger than continuum admits an uncountable biorthogonal system cf. Theorem 2.1 of \cite{sur}).
If at all possible, the construction  for density equal to any consistent value 
of the continuum  would not be easy, as the examples in the literature of Banach spaces 
which do not admit uncountable biorthogonal sets have reached only the density $\omega_2$ so far
(\cite{christina}).
Note also that it remains open if there are (even consistent) Banach spaces
(or even renormings of $\ell_1(\kappa)$) of densities in the interval $(2^\omega, 2^{(2^\omega)}]$ which do not admit
infinite or uncountable equilateral sets.

\section{Preliminaries and notation}\label{prelim}

The notation and terminology are standard.  The notation $A^B$ represents the set of all functions form
a set $B$ into a set $A$. Given a set $A$ by $[A]^2$ we mean the collection of all two-element
subsets of $A$. When $f:[A]^2\rightarrow B$, we say that $A'\subseteq A$ is $b$-monochromatic
for $b\in B$ if $f[[A']^2]=\{b\}$; a set is called monochromatic if it is $b$-monochromatic for some 
$b\in B$. The symbols $\omega_1$, $\omega_2$
denote the first and the second uncountable cardinals respectively. The
set of all natural numbers and the set of all rational numbers are
denoted by $\N$ and $\Q$ respectively.

All Banach spaces considered here are over the reals.
Whenever $(\X, \|\ \|_\X)$ is a Banach space  
we will refer only to $\X$  if $\|\ \|_\X$ is clear from the context.
We will consider norms
$\|\ \|_1$ and $\|\ \|_2$ defined as usual for a real sequence $(x(i))_{i\in \N}$ by
$\|x\|_1=\sum_{i\in \N} |x(i)|$ and $\|x\|_2=\sqrt{\sum_{i\in \N} x(i)^2}$.
By $\ell_1(A)$ for a set $A$ we mean all functions $x\in\R^{A}$ such that
$\|x\|_1<\infty$ with $\| \ \|_1$ norm and by $\ell_2$ as all functions
$x\in\R^\N$ such that $\|x\|_2<\infty$ with $\|\ \|_2$ norm.
The dual space to $\ell_1([0,1])$ is $\ell_\infty([0,1])$
together with the action 
$$\langle\phi, x\rangle=\sum_{t\in [0,1]}\phi(t)x(t)$$
for any $\phi\in \ell_\infty([0,1]$ and $x\in \ell_1([0,1])$.
By a support of $x\in \ell_1(A)$ we mean $\{a\in A: x(a)\not=0\}$; it is denoted $supp(x)$.
If $x\in \ell_1(A)$ and $B\subseteq A$, then by $x|B$ we mean the coordinatewise product
of $x$ and the characteristic function of $B$. 

Recall that a Banach  space $(\X, \|\ \|_\X)$ is strictly convex
if $\|x+y\|=\|x\|+\|y\|$ implies that there is $\lambda>0$ such that $x=\lambda y$
for any $x\not=0\not=y$. It is well known that the norm on $\ell_2$ is strictly convex.
We say that two norms $\|\ \|$ and $\|\ \|'$ on a Banach space $\X$
are equivalent if there are constants $c, C>0$ such that
$c\|x\|\leq\|\ x\|'\leq C\|\  x\|$ for every $x\in \X$. This is equivalent to the fact that
the identity is an isomorphism between $(\X, \|\ \|)$ and $(\X, \|\ \|')$.

\begin{lemma} Suppose that $\X, \Y$ are Banach spaces and $T: \X\rightarrow \Y$
is a bounded linear operator. Then the norm $\|\ \|_T$ on $\X$ given by
$\|x \|_\X+\|T(x)\|_\Y$ for $x\in \X$ is equivalent to the norm $\|\ \|_\X$.
If $T$ is injective and $\Y$ is strictly convex, then $\|\ \|_T$ is strictly convex.
\end{lemma}
\begin{proof} We have 
$$\|x\|_\X\leq \|x\|_\X+\|T(x)\|_\Y\leq (1+\|T\|)\|x\|.$$
For strict convexity suppose that 
$\|x+y\|_T=\|x\|_T+\|y\|_T$ for some nonzero $x, y\in \X$. So
$\|x+y\|_\X+\|T(x+y)\|_\Y=\|x\|_\X+\|y\|_\X+\|T(x)\|_\Y+\|T(y)\|_\Y$. By the triangle 
inequality this means that $\|x+y\|_\X=\|x\|_\X+\|y\|_\X$
and $\|T(x+y)\|_\Y=\|T(x)\|_\Y+\|T(y)\|_\Y$. The injectivity of $T$ yields $T(x)\not=0\not=T(y)$.
The strict convexity of $\Y$ yields
$\lambda>0$ such that $T(x)=\lambda T(y)$. The injectivity of $T$ gives 
$x=\lambda y$.

\end{proof}

Let $(I_i)_{i\in \N}$ be an enumeration of all subintervals
of $[0,1]$ with rational end-points. We define $x^*_i=\chi_{I_i}\in \ell_\infty([0,1])$,
where $ \chi_{I_i}$ is the characteristic function of $I_i$. 
Given a nonzero $x\in \ell_1([0,1])$ we find $t\in [0,1]$ such that
$x(t)\not=0$ and an open interval $I_i\ni t$ such that
$\sum\{|x(t')|: t'\in I_i\setminus \{t\}\}<|x(t)|$. Then $\langle x^*_i, x\rangle\not=0$.
This shows that $\{x^*_i: i\in \N\}$ is total for $\ell_1([0,1])$ i.e.,
$x_i^*(x)=0$ for each $i\in I$ implies that $x=0$.
Observe that $\|x^*_i\|=1$ for each $i\in \N$.

\begin{lemma}\label{operators} There is a bounded compact injective operator
$T:\ell_1([0,1])\rightarrow \ell_2$.
\end{lemma}
\begin{proof}
Define  $T$ by
$$T(x)=\Big({{x^*_i(x)}\over{2^i}}\Big)_{n\in \N}$$
for any $x\in \ell_1([0,1])$. As $x^*_i$s form a total set,
the operator is injective.  It is also clear that the values of $T$ are in $\ell_2$
and the operator is bounded with its norm $\sqrt2$, as $\|x_i^*\|=1$ or each $i\in \N$.
For the compactness, use again the fact that the norms of $x_i^*$s are $1$
and so $T$ can be approximated in the operator norm by finite rank operators
which are $T$ up to the $k$-th coordinate and later $0$ for $k\in \N$.
As compact operators form a closed ideal this proves the compactness of $T$.
\end{proof}

\section{Proof of the main result}

\noindent{\bf Theorem 1.} {\it Suppose that $\X$ is a Banach space with a strictly convex norm and 
that $T:\ell_1([0,1])\rightarrow \X$ is a compact bounded injective operator.
Then the equivalent renorming $(\ell_1([0,1]), \|\ \|_T)$
of $(\ell_1([0,1]), \|\ \|_1)$ 
admits no infinite equilateral set.}

\begin{proof}
Suppose that $\{x_n: n\in \N\}$ is equilateral in $\ell_1([0,1])$
with the norm $\|\ \|_T$. We will derive a contradiction.
 By scaling it, we may assume that it is $1$-equilateral.
As the supports of $x_n$'s are countable, they are all included in some countable
$A\subseteq [0,1]$.  So we need to prove that the corresponding renorming of the
separable $\ell_1(A)$ does not admit an infinite equilateral set.

By the compactness of $T$,  passing to a subsequence we may assume that 
$\{T(x_n): n\in \N\}$ converges in the norm $\|\ \|_\X$ to $z\in \X$. 
As the range of $T$ is not closed, $z$ does not need to belong to it.
Since we work now with
separable $\ell_1(A)$ and  $(x_n)_{n\in \N}$ is bounded, 
by passing to a subsequence we may assume that $(x_n)_{n\in \N}$ 
converges pointwise to $y\in \ell_1(A)$, that is for every $t\in A$
the sequence $(x_n(t))_{n\in \N}$ converges to $y(t)$. 

Let $x_n'=x_n-y$ for every $n\in \N$.
Then $\|x_n'-x_m'\|_T=\|x_n-x_m\|_T=1$ for all distinct $n, m\in \N$ and  for every $t\in A$
the sequence $(x_n'(t))_{x\in \N}$ converges to $0$.
Moreover
$(T(x_n'))_{n\in \N}$ converges in the norm $\|\ \|_\X$ to $z'=z-T(y)$.

Fix $m\in \N$ and $\varepsilon>0$. Choose a finite $F\subseteq A$ such
that $\|x_m'|F-x_m'\|_1<\varepsilon/4$. As $(x_n'(t))_{n\in \N}$ converges to $0$,
for each $t\in F$, for sufficiently large $n\in \N$ we have $\|x_n'|F\|_1\leq\varepsilon/4$ and so
we have
\begin{itemize}
\item $\|x_n'\|_1\leq \|x_n'-x_n'|F\|_1+\varepsilon/4$,
\item $|x_m'\|_1\leq \|x_m'|F\|_1 +\varepsilon/4$,
\item $\|x_n'-x_m'\|_1\geq \|(x_n'-x_n'|F)- x_m'|F\|-\varepsilon/2$
\item $\|(x_n'-x_n'|F)-x_m'|F\|_1=\|(x_n'-x_n'|F)\|_1+\|x_m'|F\|_1$ as these vectors
have disjoint supports,
\end{itemize}
so we obtain
$$\|x_n'\|_1+\|x_m'\|_1-\|x_n'-x_m'\|_1\leq$$
$$\leq \|x_n'-x_n'|F\|_1+\|x_m'|F\|_1-\|(x_n'-x_n'|F)- x_m'|F\|+\varepsilon=\varepsilon.$$
So for every $m\in \N$ we have 
$$lim_{n\rightarrow \infty}\big(\|x_n'\|_1+\|x_m'\|_1-\|x_n'-x_m'\|_1\big)=0.\leqno (*)$$
For $n\in \N$ define
$$c_n=1/2-\|x_n'\|_1.$$

\begin{claim}\label{one-half}
$\lim_{n\rightarrow\infty}c_n=0$.
\end{claim}

\noindent{\it Proof of the Claim:}
Since the sequence of $c_n$s is bounded (as the $x_n'$s form an equilateral set), 
by passing to a subsequence we may assume that it is
converging to $c$. First suppose that $c>0$. Let $k\in \N$ and $\varepsilon>0$
be such that $\|x_n'\|<1/2-\varepsilon$ for all $n>k$. Then by passing to a subsequence
we may assume that $\|T(x_n')-T(x_m')\|_\X\leq\varepsilon$ for all $n, m\in \N$ as
$(T(x_n'))_{n\in \N}$ converges to $z'$ in $\X$. Fixing $m>k$
by the triangle inequality we obtain
$$\|x_n'-x_m'\|_T=\|x_n'-x_m'\|_1+\|T(x_n'-x_m')\|_\X \leq\|x_n'\|_1+\|x_m'\|_1+ \|T(x_n')-T(x_m')\|_\X<$$
$$2(1/2-\varepsilon)+\varepsilon\leq 1$$
contradicting the fact that $x_n'$s form a $1$-equilateral set.

Now suppose that $c<0$. Let $k\in \N$ and $\varepsilon>0$
be such that $\|x_n'\|>1/2+\varepsilon$ for all $n>k$. Fixing $m>k$
by (*) we may find $n\in \N$ such that $\|x_n'\|_1+\|x_m'\|_1-\|x_n'-x_m'\|_1\leq\varepsilon$.
So
$$\|x_n'-x_m'\|_T>\|x_n'-x_m'\|_1=(\|x_n'-x_m'\|_1-\|x_n'\|_1-\|x_m'\|_1)+(\|x_n'\|_1+\|x_m'\|_1)> $$
$$-\varepsilon+2(1/2+\varepsilon)\geq 1$$
contradicting the fact that $x_n'$s form a $1$-equilateral set. This completes the proof of the claim.

For distinct $m, n\in \N$ we have
$$\|x_n'-x_m'\|_T=\|x_n'-x_m'\|_1+\|T(x_n')-T(x_m')\|_\X=1,$$
so subtracting $1=1/2+1/2=(\|x_n'\|_1+c_n)+(\|x_m'\|_1+c_m)$ from both sides of the second equality, we get
$$-c_n-c_m-\|x_n'\|_1-\|x_m'\|_1+\|x_n'-x_m'\|_1+\|T(x_n')-T(x_m')\|_\X=0.\leqno (**)$$
Fixing any $m\in\N$, by going to infinity with $n\in\N$ by (*) and Claim \ref{one-half} we obtain
$$\|z'-T(x_m')\|_\X=c_m.\leqno (***)$$
Defining $u_m=T(x_m')-z'$ and combining (**) and (***) we obtain for distinct $n,m\in \N$
$$\|x_n'-x_m'\|_1-\|x_n'\|_1-\|x_m'\|_1=\|u_n\|_\X+\|u_m\|_\X-\|u_n-u_m\|_\X.$$
By the triangle inequality the right hand side of the above is non-negative while the left hand side
is non-positive which implies that both of the expressions are equal to zero. In
particular, the left hand side is zero.

Note that $u_n\not=u_m$ for distinct $n, m\in \N$ because otherwise we would have
$T(x_n')=T(x_m')$, which implies $x_n'=x_m'$ by the injectivity of $T$ and this contradicts 
the fact that $x_n'$s form a $1$-equilateral set. So at most one $u_n$ can be zero.
By passing to an infinite subset we may assume that all are nonzero, so that we can
apply the definition of the strict convexity.

By the strict convexity of the norm in $\X$ we obtain $\lambda_{m, n}>0$
such that $u_m=\lambda_{m, n} u_n$.  If $\lambda_{m, n}=1$ for some distinct $m, n\in \N$,
then $T(x_n')=T(x_m')$ which contradicts the injectivity of $T$. 
Otherwise $(z'-T(x_m'))=\lambda_{m, n}(z'-T(x_n'))$ which gives
$z'(1-\lambda_{m, n})=-\lambda_{m, n}T(x_n')+T(x_m')$ and so
$$z'=T\Big( {1\over{1-\lambda_{m, n}}}\big(  x_m'-\lambda_{m, n}x_n'\big)\Big)$$
By the injectivity of $T$ this means for any distinct $k, l\in \N$
$${1\over{1-\lambda_{m, n}}}\big(  x_m'-\lambda_{m, n}x_n'\big)={1\over{1-\lambda_{k, l}}}\big(  x_k'-\lambda_{m, n}x_l'\big),$$
 in particular that 
 $$x_m'={{1-\lambda_{m, n}}\over{1-\lambda_{k, l}}}\big(  x_k'-\lambda_{m, n}x_l'\big)+\lambda_{m, n}x_n'.$$
That means that $\{x_n': n\in \N\}$ spans a two dimensional space. However
 such a space cannot admit an infinite equilateral set as its ball is compact and so any bounded sequence
 has a convergent subsequence. This is the required contradiction.

\end{proof}

\begin{corollary}\label{dushnik}
 Suppose that $\X$ is a Banach space with a strictly convex norm and 
that $T:\ell_1([0,1])\rightarrow \X$ is a compact bounded injective operator.
Then the equivalent renorming $(\ell_1([0,1]), \|\ \|_T)$
of $(\ell_1([0,1]), \|\ \|_1)$ has the following property: Any infinite  $r$-separated  subset
$\Y\subseteq \ell_1([0,1]$  has a subset $\Z\subseteq \Y$ of the same cardinality which is $(r+)$-separated.
\end{corollary}
\begin{proof} Define a function $c: [\Y]^2\rightarrow\{0,1\}$
by putting $c(\{y, y'\})=0$ if and only if $\|y-y'\|_T=r$.
First consider the case when $\Y$ is countable.
By Ramsey's (Problem 24.1 of \cite{komjath})
theorem there is an infinite  monochromatic subset of $\Y$. It cannot be
$0$-monochromatic as $(\ell_1([0,1]), \|\ \|_T)$ has no infinite equilateral 
sets by Theorem \ref{infinite}. A $1$-monochromatic infinite subset is the required one.
Now consider an uncountable regular cardinality of $\Y$.
A version of Dushnik Miller theorem says that for any uncountable cardinal
$\kappa$ for any $f: [\kappa]^2\rightarrow\{0,1\}$ there is
either an infinite $0$-monochromatic set or a $1$-monochromatic subset
of cardinality $\kappa$ (Problem 24.13 of \cite{komjath}). So apply this to $c$ and use the fact that
$(\ell_1([0,1]), \|\ \|_T)$ has no infinite equilateral sets by Theorem \ref{infinite}.
A $1$-monochromatic infinite subset is the required one.
\end{proof}

\section{Renormings induced by injective separable range operators}

In this section we prove that a substantial part of the property of Corollary \ref{dushnik} holds 
for much bigger class of renormings of $\ell_1([0,1])$ than
those  considered in Theorem \ref{infinite}. 
\vskip 6pt

\noindent{\bf Theorem 3.}
{\it Suppose that $\X$ is a Banach space and  $T:\ell_1([0,1])\rightarrow \X$ is injective 
and has separable range and $r>0$.
Then $(\ell_1([0,1]), \|\ \|_T)$
 has the following property: Any $r$-separated subset  $\Y\subseteq \ell_1([0,1])$ of regular uncountable cardinality has a subset $\Z\subseteq \Y$ of the same cardinality which is $(r+)$-separated.
In particular $(\ell_1([0,1]), \|\ \|_T)$ does not admit any uncountable equilateral set..}

\begin{proof}

Let $\{d_n: n\in \N\}$ be a dense subset of the range of $T$.
Suppose that $\kappa$ is a regular uncountable cardinal
and  $\Y=\{x_\alpha: \alpha<\kappa\}\subseteq \ell_1([0,1])$ is $r$-separated in the norm
$\|\ \|_T$.
As the supports of $x_\alpha$s are countable, their union has cardinality at most $\kappa$.
So we may assume that $\Y\subseteq \ell_1(A)$ for some $A\subseteq[0,1]$ of cardinality  $\kappa$.
Let $A=\{t_\xi: \xi<\kappa\}$.

Now we need to define certain function $M$. The domain of $M$ will consist of $5$-tuples
of the form 
$(\varepsilon, q,  F, s, n)$, where $\varepsilon, q>0$ are rationals, $F$ is a finite subset 
of $[0,1]$, $s\in \Q^F$ and $n\in\N$.
Note that there are only countably many such $\varepsilon, q, n$ and given a finite 
$F\subseteq [0,1]$
there are only countably many choices for $s\in \Q^F$. 

The function $M$ will assume values in $\kappa$. It is defined as follows: if there is $\alpha<\kappa$
such that 
\begin{enumerate}[(a)]
\item 
 $\|d_n-T(x_{\alpha})\|_\X<\varepsilon/10,$
\item 
$|\|x_{\alpha}|([0,1]\setminus F)\|_1-q|<2\varepsilon/10$,
\item 
$\|s-(x_{\alpha}|F)\|_1<\varepsilon/10$,
\end{enumerate}
then we choose minimal such $\alpha$ and define $M(\varepsilon, q,  F, s, n)$ as the minimal
ordinal less than $\kappa$ such that $supp(x_\alpha)\subseteq 
\{t_\xi: \xi<M(\varepsilon, q,  F, s, n)\}$.
If there is no such $\alpha$, then we define $M(\varepsilon, q,  F, s, n)$ anyhow.

\begin{claim} 
$$C=\{\delta<\kappa: \forall  (\varepsilon, q,  F, s, n)\in dom(M) 
[F\subseteq\{t_\xi:\xi<\delta\}\ \Rightarrow \ M(\varepsilon, q,  F, s, n)<\delta]\}$$
is unbounded in $\kappa$
\end{claim}
\noindent{\it Proof of the Claim.}
Fix  $\delta_0<\kappa$. By recursion define a strictly increasing $(\delta_n)_{n\in \N}$
in $\kappa$ such that  $M(\varepsilon, q,  F, s, n)<\delta_{n+1}$ whenever
$(\varepsilon, q,  F, s, n)$ is in the domain of $M$ and $F\subseteq\delta_n$.
Given $\delta_n$ there are less than $\kappa$ many elements in the domain of $M$
such that $F\subseteq\delta_n$ as there are only less than $\kappa$  such $F$s, so
by the regularity of $\kappa$ the next
$\delta_{n+1}$ can be taken as the supremum of  all the values under $M$ of such elements.
One sees that $\delta=\sup\{\delta_n: n\in \N\}$ is in $C$. As $\delta_0$ was arbitrary, 
this completes the proof of the Claim.

For any $\delta\in C$ and $\alpha<\kappa$ we define
$$x_{\alpha, \delta}=x_{\alpha}|\{t_\xi: \delta\leq \xi<\kappa\}$$

\begin{claim}\label{model} If $x_{\alpha, \delta}\not=0$, then $\|x_{\alpha, \delta}\|_1\geq r/2$ for each $\alpha<\kappa$
and $\delta\in C$.
\end{claim}
\noindent{\it Proof of the Claim.}
Fix $\delta\in C$ and $\alpha<\kappa$
such that $x_{\alpha, \delta}\not=0$. 
Fix a rational $\varepsilon>0$. We will show that there is $\beta<\alpha$
such that $|\|x_\beta-x_{\alpha}\|_T-2\|x_{\alpha, \delta}\|_1|<\varepsilon$. As 
$\{x_\alpha: \alpha<\kappa\}$ is an $r$-separated set and $\varepsilon>0$ is arbitrary, this 
is sufficient.
Find 
\begin{enumerate}
\item $q\in \Q$ such that $|q-\|x_{\alpha, \delta}\|_1|<\varepsilon/10$,
\item $n\in \N$ such that 
 $$\|d_n-T(x_{\alpha})\|_\X<\varepsilon/10,$$
\item finite $F\subseteq\{t_\xi: \xi<\delta\}$ such that
$|\|x_{\alpha}|([0,1]\setminus F)\|_1-q|<2\varepsilon/10$,
\item $s\in \Q^F$ such that 
$\|s-(x_{\alpha}|F)\|_1<\varepsilon/10$,
\end{enumerate}
So $x_{\alpha}$ satisfies the following formulas when substituted in place of $x$;
\begin{enumerate}
\item[(5)] 
 $\|d_n-T(x)\|_\X<\varepsilon/10,$
\item[(6)]
$|\|x|([0,1]\setminus F)\|_1-q|<2\varepsilon/10$,
\item [(7)]
$\|s-(x|F)\|_1<\varepsilon/10$.
\end{enumerate}
As $F\subseteq\{t_\xi: \xi<\delta\}$ since $\delta\in C$, we have $M(\varepsilon, q,  F, s, n)<\delta$.
As there is $x_{\gamma}$ (for $\gamma=\alpha$) which satisfies (5) - (7) when substituted for $x$,
by the definition of $M$ 
there is $\beta$  such that $x_\beta$ satisfies (5) -(7) when substituted for $x$  and $x_\beta$
has its support included in $\{t_\xi: \xi<M(\varepsilon, q,  F, s, n)\}$ and in particular in 
$\{t_\xi: \xi<\delta\}$.

Now we estimate $\|x_\beta-x_{\alpha}\|_T$:
$$\|T(x_\beta-x_{\alpha})\|_\X=\|T(x_\beta)-T(x_{\alpha})\|_\X\leq \|T(x_\beta)-d_n\|_\X
+\|d_n-T(x_{\alpha})\|_\X
\leq 2\varepsilon/10\leqno(8)$$
by (2) and (5) for $x_\beta$ in place of $x$.
$$\|x_\beta-x_{\alpha}\|_1=\|x_\beta|F-x_{\alpha}|F\|_1+
\|x_\beta|(\{t_\xi: \xi<\delta\}\setminus F)-x_{\alpha}|(\{t_\xi: \xi<\delta\}\setminus F)\|_1
+\|x_{\alpha, \delta}\|_1\leqno(9)$$
since the support of $x_\beta$ is included in $\{t_\xi: \xi<\delta\}$.
Conditions (4) and  (7) for $x_\beta$ in place of $x$ imply that
$$\|x_\beta|F-x_{\alpha}|F\|_1\leq 2\varepsilon/10.\leqno (10)$$
Conditions (1) and (3) imply that 
$$\|x_{\alpha}|(\{t_\xi: \xi<\delta\}\setminus F)\|_1< 3\varepsilon/10$$
and so by (6) for $x_\beta$ in place of $x$ and the fact that the support of $x_\beta$ is included in $\{t_\xi: \xi<\delta\}$
we conclude that
$$q-5\varepsilon/10\leq \|x_\beta|(\{t_\xi: \xi<\delta\}\setminus F)
-x_{\alpha}|(\{t_\xi: \xi<\delta\}\setminus F)\|_1\leq q+5\varepsilon/10,\leqno(11)$$
so by (1)  and (8) - (11) we conclude that
$$2\|x_{\alpha, \delta}\|_1-6\varepsilon/10\leq \|x_\beta
-x_{\alpha}\|_T\leq 2\|x_{\alpha, \delta}\|_1+\varepsilon,$$
which completes the proof of the Claim.

Now note that 
as $\{x_\alpha: \alpha<\kappa\}$ is discrete, it cannot be contained 
in any subspace of the form $\ell_1(\{t_\xi: \xi<\delta\})$ for $\delta<\kappa$ which has
density less than $\kappa$. So for  $\delta\in C$ we can find  $\alpha_\delta<\kappa$
such that $x_{\alpha_\delta, \delta}\not=0$ and moreover we may make sure that
$\alpha_\delta\not=\alpha_{\delta'}$  for any $\delta<\delta'$ in $C$.
Next we find $C'\subseteq C$ of cardinality $\kappa$ such that 
$$supp(x_{\alpha_{\delta'}, {\delta'}})\subseteq\{t_\xi: \delta'\leq \xi<\delta\}\leqno (12)$$
for any $\delta'<\delta$ and $\delta, \delta'\in C'$.
This can be done by recursion  taking at the inductive step the next $\delta\in C$ such that
the supports of the previous $x_{\alpha_{\delta'}, \delta'}$s are included
in $\{t_\xi: \xi<\delta\}$.

Now we will consider two cases, the first when there is $\theta\in C'$ such
 that for any $\delta, \delta'\in C'$ such that 
 $\theta<\delta'<\delta$ we have 
 $$supp(x_{\alpha_{\delta'},\delta'})\cap supp(x_{\alpha_{\delta}})=\emptyset.$$
 Then  by Claim \ref{model} we have 
 $$\|x_{\alpha_{\delta}}-x_{\alpha_{\delta'}}\|_1\geq 
 \|x_{\alpha_{\delta'},\delta'}\|_1+\|x_{\alpha_{\delta},\delta}\|_1\geq r$$
 for every $\theta<\delta'<\delta<\kappa$. Since $T$ is injective
  $\|x_{\alpha_{\delta}}-x_{\alpha_{\delta'}}\|_T>\|x_{\alpha_{\delta}}-x_{\alpha_{\delta'}}\|_1\geq r$,
  so we obtain that $\Z=\{x_{\alpha_\delta}: \delta\in C', \ \delta>\theta\}$.
  
  The second case is when there is no $\theta\in C'$ as in the first case. Then 
  by recursion we can construct $C''\subseteq C'$ and $(\theta_\delta: \delta\in C'')$
  such that 
  $$sup\{\xi: t_\xi\in supp(x_{\alpha_{\delta'}})\}< \theta_\delta<\delta, \ \ x_{\alpha_{\delta}}(t_{\theta_\delta})\not=0\leqno (13)$$
  for every $\delta'<\delta$ with $\delta, \delta'\in C''$. 
  Indeed, having constructed less then $\kappa$ elements $\delta'\in C''$ whose supremum is $\theta< \kappa$
  we consider all $\delta\in C'$ which are above $\theta$. Since there is no $\theta\in C'$ as
  in the first case we find $\theta<\delta''<\delta$ with $\delta'', \delta\in C'$ such that
  $supp(x_{\alpha_{\delta''},\delta''})\cap supp(x_{\alpha_{\delta}})\not=\emptyset$.
  To complete the recursion take $\delta$ as the next element of
  $C''$ and take as $t_{\theta_\delta}$  any element of the above intersection.
  
  Since
  $\kappa$ is a  cardinal of uncountable cofinality, by passing to 
  a subset of cardinality $\kappa$ we may assume that there is $\varepsilon>0$ 
  such that $|x_{\alpha_{\delta}}(t_{\theta_\delta})|\geq\varepsilon$ for every $\delta\in C''$.
  
  Now we will use the following version of the Dushnik-Miller theorem: if
  $\kappa$ is a regular uncountable cardinal and  $c:[\kappa]^2\rightarrow\{0,1\}$,
  then either there is a $0$-monochromatic set for $c$ which has its order type equal to $\omega+1$
  or there a $1$-monochromatic subset of cardinality $\kappa$ for $c$ (Problem 24.32 of \cite{komjath}).
  We define $c: [C'']^2\rightarrow\{0,1\}$ by $c(\delta', \delta)=0$ for $\delta'<\delta$
  if and only if 
  $$\|x_{\alpha_{\delta}}-x_{\alpha_{\delta'}}\|_T=r.$$
  So it is enough to prove that there is no $0$-monochromatic set of order type $\omega+1$.
  Suppose $\{\delta_n: n\in \N\}\cup\{\delta_\omega\}$ forms such a set, where
  $\delta_\xi<\delta_\eta$ if $\xi<\eta$ for all $\xi, \eta< \omega+1$. 
  Since the supports of $x_{\alpha_{\delta},\delta}$s for $\delta\in C''$
  are pairwise disjoint by (12) there is $n\in \N$ such that 
  $$\|x_{\alpha_{\delta_\omega}}| supp(x_{\alpha_{\delta_n},\delta_n})\|_1\leq \varepsilon/2.$$
  Also by (13) we have  $t_{\theta_{\delta_\omega}}\not\in supp(x_{\alpha_{\delta_n}, \delta_n})\cup 
  supp(x_{\alpha_{\delta_\omega}, \delta_\omega})$, where the union is disjoint by (12). 
  So by Claim \ref{model} we have 
  $$\|x_{\alpha_{\delta_\omega}}-x_{\alpha_{\delta_n}}\|_T>
  \|x_{\alpha_{\delta_\omega}}-x_{\alpha_{\delta_n}}\|_1\geq
  \|x_{\alpha_{\delta_\omega}, \delta_\omega}\|_1+|x_{\alpha_{\delta_\omega}} 
  (t_{\theta_{\delta_\omega}})|+\|x_{\alpha_{\delta_n}, \delta_n}\|_1-\varepsilon/2
  \geq r$$
  which contradicts the choice of $\{\delta_n: n\in \N\}\cup\{\delta_\omega\}$ as $0$-monochromatic.
  So the set of vectors $x_{\alpha_\delta}$, for $\delta$ in the $1$-monochromatic set of cardinality 
  $\kappa$, is the desired $\Z$.
\end{proof}

\begin{remark}\label{plus-id}
Some of the renormings considered in Theorem \ref{uncountable} admit many infinite
equilateral sets. For example we can identify $\ell_1([0,1])$ with
$\ell_1\oplus\ell_1([0,1])$ and define $T': \ell_1\oplus\ell_1([0,1])\rightarrow \ell_1\oplus\ell_2$
by $T'(x, y)=(x, T(y))$, where $T$ is as in Lemma \ref{operators}. 
\end{remark}

\begin{remark}\label{sierpinski} Let us remark that the property of the spaces from Theorem \ref{uncountable} is much stronger
than not admitting uncountable equilateral sets. To see this consider
$\ell_\infty([[0,1]]^2)$  with the usual $\|\ \|_\infty$ norm. Let 
$c:[[0,1]]^2\rightarrow \{0, 1\}$ be Sierpi\'nski's colouring with no
uncountable monochromatic subset (Problem 24.23 of \cite{komjath}). 
So $c\in \ell_\infty([[0,1]]^2)$ and consider $f_t\in \ell_\infty([[0,1]]^2)$
defined by 

$$
  f_t(\{r,s\}) =
  \begin{cases}
    c(\{r, s\}) & \text{if $t=\min(\{r,s\})$} \\
     -c(\{r, s\}) & \text{if $t=\max(\{r,s\})$} \\
    0 & \text{otherwise.}
  \end{cases}
$$
For distinct $t, t'\in [0,1]$ the intersection of supports of $f_t$ and $f_{t'}$ 
is included in $\{\{t, t'\}\}$. For $t<t'$ the value of $f_t-f_{t'}$ at $\{t, t'\}$
is $1-(-1)=2$ if $c(\{t, t'\})=1$ and it is $0$ otherwise, so
$$
  \|f_t-f_{t'}\| =
  \begin{cases}
    1 & \text{if $c(\{t,t'\})=0$} \\
    2 & \text{if $c(\{t,t'\})=1$.}
  \end{cases}
$$
This  means that $\{f_t: t\in [0,1]\}$  is a $1$-separated subset of
the unit sphere which does not admit any  uncountable
equilateral subset but also there is  no uncountable subset
which is $(1+)$-separated. Note that any countable
$1$-separated set $\{x_n: n\in \N\}$ in any Banach space, 
 contains either an infinite $1$-equilateral set or an infinite
 $(1+)$-separated set by Ramsey's theorem (Problem 24.1 of \cite{komjath}).
\end{remark}

\bibliographystyle{amsplain}

\end{document}